\documentclass[12pt]{amsart}
\usepackage[english]{babel}
\usepackage[utf8]{inputenc}
\usepackage{lipsum}
\usepackage{mathrsfs}
\usepackage{stix}
\usepackage{fullpage}
\usepackage{mathtools}
\usepackage{amsmath}
\usepackage{tikz}
\usepackage{tikz-cd}
\usetikzlibrary{decorations.pathreplacing,angles,quotes}
\usetikzlibrary{arrows,chains,positioning,scopes,quotes}
\newtheorem{defi}{Definition}[section]

\newtheorem{lema}[defi]{Lemma}
\newtheorem{teo}[defi]{Theorem}
\newtheorem{rem}[defi]{Remark}

\newtheorem{pro}[defi]{Proposition}
\newtheorem*{innerthm}{\innerthmname}
\newcommand{\innerthmname}{}
\newenvironment{statement}[1]
 {\renewcommand{\innerthmname}{#1}\innerthm}
 {\endinnerthm}
\theoremstyle{definition}

\newtheorem*{rem*}{Remark}

\newcommand{\1}{\mathbb{1}}

\newcommand{\interior}[1]{%
  {\kern0pt#1}^{\mathrm{o}}%
}

\newcommand{\T}{\mathbb{T}}

\newcommand{\C}{\mathbb{C}}
\newcommand{\Q}{\mathbb{Q}}

\newcommand{\R}{\mathbb{R}}
\newcommand{\N}{\mathbb{N}}
\newcommand{\Z}{\mathbb{Z}}

\newcommand{\esp}{\text{  }}

\usepackage[colorlinks=false]{hyperref}
\renewcommand\eqref[1]{(\ref{#1})} 
\begin{document}
\title[Titchmarsh Theorems for H{\" o}lder-Lipschitz functions on profinite groups] 
 {Titchmarsh Theorems for H{\" o}lder-Lipschitz functions on profinite groups}

\author{ J.P. Velasquez-Rodriguez
}

\newcommand{\Addresses}{
 {
  \bigskip
  \footnotesize
  
  J.P. Velasquez-Rodriguez, \textsc{Department of Mathematics: Analysis, Logic and Discrete Mathematics, Ghent University, Belgium, and Departamento de Matematicas, Universidad del Valle, Cali-Colombia}\par\nopagebreak
  \textit{E-mail address:} \texttt{juanpablo.velasquezrodriguez@ugent.be / velasquez.juan@correounivalle.edu.co}
}

}

\thanks{The author is supported by the FWO Odysseus 1 grant G.0H94.18N: Analysis and Partial Differential Equations.}



\date{\today}
\begin{abstract}
In this note we extend to metrizable profinite groups the classical theorems of Titchmarsh on the Fourier transform of Hölder–Lipschitz functions. This generalizes the results of Younis on compact zero-dimensional abelian groups to the noncommutative case, and proves a relation between the Hölder–Lipschitz-continuity of functions and their "Sobolev regularity" given in terms of the Vladimirov-Taibleson operator. Since the class of profinite groups is fairly big, the formulation of our results requires to impose a special condition on the representation theory of the group. We prove that in particular such condition is satisfied by compact nilpotent metrizable profinite groups, which covers the case of compact nilpotent $\ell$-adic Lie groups. In addition, we study the modulus of continuity of $L^2$-functions on the group, the functional spaces related to it, and its relation to the $L^2$-based Hölder–Lipschitz-spaces. Finally, we also derive a characterisation for Dini–
Lipschitz classes on metrizable profinite groups in terms of the behaviour of
their Fourier coefficients.
\end{abstract}
\maketitle
\tableofcontents
\section{Introduction}
In this note our main goal is to extend to noncommutative metrizable profinite groups the following two Titchmarsh theorems for H\"older-Lipschitz functions: 

\begin{statement}{First Titchmarsh Theorem}\label{TeoTitchmarsh1}
Let $G$ be an abelian metrisable profinite group. Let $0 < \alpha \leq 1$ and $1 < p \leq 2$. If $f \in Lip_{G }(\alpha ; p)$, then we have $\widehat{f} \in L^\beta (\widehat{G})$ for $ \frac{p}{p + \alpha p -1} < \beta \leq q$, $1/p + 1/q =1$. Here the Lipschitz space $Lip_{G }(\alpha ; p)$ is defined as: $$Lip_{G }(\alpha ; p) := \{f \in L^p(G): \esp \| f(h + \cdot) - f(\cdot) \|_{L^p (G)}= \mathcal{O} (|h|^\alpha) \esp \text{as } \esp h \to 0 \}, \esp \esp 0<\alpha \leq 1, \esp \esp 1 < p \leq \infty.$$
\end{statement}

\begin{statement}{Second Titchmarsh Theorem}\label{TeoTitchmarsh2}
Let $0 < \alpha \leq 1$ and $f \in L^2(G)$. Then $f \in Lip_{G }(\alpha ; 2) $ if and only if $$\sum_{[\xi] \in \widehat{G}\setminus G_k^\bot} |\widehat{f}(\xi)|^2 = \mathcal{O}(|G/G_k|^{-2\alpha}),$$for any $k \in \N_0$. Here $H^\bot$ denotes the anihilator of $H \leq G$, and $\{G_n\}_{n \in \N_0}$ in a strictly decreasing sequence of compact open subgroups that form a basis of neighbourhoods at the identity. 

\end{statement}
These theorems, initially proven by Titchmarsh in $\R^n$, were later extended to compact abelian groups by M. Younis in his PhD dissertation \cite{Younisthesis}. Since then, numerous authors have attempted to replicate these findings in a variety of contexts where Fourier analysis is available. For instance, just to give some recent examples, R. Daher and collaborators have conducted a systematic exploration of Titchmarsh's results for transforms such as the Dunkl transform, Jacobi-Dunkl transform, and generalized Fourier-Bessel transform, among others. See \cite{Daher1, Daher4, Daher3, Daher2} for further details. In recent years, M. El Hamma and A. Mahfoud treated the $L^p$-case for the first Hankel-Clifford transform \cite{LpHankel-Clifford}, while H. Monaim and S. Fahlaoui investigated Titchmarsh-type theorems for the quaternionic Fourier transform \cite{Monaim2022}.    

In the totally disconnected case, S. Platonov extended Titchmarsh theorems to locally compact totally disconnected groups, which are sometimes called zero-dimensional groups. He proved an analogue of the second Titchmarsh theorem, first on the field of $\ell$-adic numbers $\Q_\ell$, and later on locally compact abelian Vilenkin groups\cite{Platonov1, Platonov2}. Here $\ell$ is a prime number grater than 2 that we fix for the rest of the paper. Our work, on the other hand, focuses on compact groups much like in \cite{DaherDelgadoRuzhansky}. When studying compact groups, one mainly considers two classes: locally connected compact groups, such as any compact real Lie group, and totally-disconnected compact groups, such as metrizable profinite groups, like compact Lie groups over non-archimedean local fields. Younis considered both types of groups in his doctoral thesis and proved how the Lipchitzness of a function is related to its Fourier coefficients when the group is abelian. However, in the noncommutative case things become more complicated, and the question of the relationship between the Lipschitzness of a function and its group Fourier transform is still interesting. For locally connected groups, Younis's results have already been improved upon in \cite{DaherDelgadoRuzhansky}, where the authors extended Titchmarsh theorems to compact homogeneous manifolds, and in particular to compact Lie groups. That covers the locally connected case, and gives an idea on what to expect on other classes of compact groups, but omits the totally disconnected case, which is the subject of this paper. We will focus on compact Vilenkin groups, a class of metrizable topological groups that includes most examples of profinite groups, including compact Lie groups over non-archimedean local fields.

\begin{defi}\label{defilcvilenkingroup}\normalfont
We say that a topological group $G$ is a \emph{compact Vilenkin group} if $G$ is a profinite group endowed with a strictly decreasing sequence of compact open subgroups $\mathscr{G}:=\{G_n\}_{n \in \Z}$ such that
\begin{enumerate}
    \item[(i)] It holds:$$2 \leq \varkappa_n := | G_n /G_{n+1}|< \infty,$$ for every $n \in \N_0.$
    \item[(ii)]$$G= \bigcup_{n \in \N_0} G_n , \esp \esp \text{and} \esp \esp \bigcap_{n \in \N_0} G_n = \{e\}.$$
    \item[(iii)]The sequence $\{G_n\}_{n \in \N_0}$ form a basis of neighbourhoods at $e \in G$. 
    \end{enumerate}
    If we also have $$\sup_{n \in \N_0} |G_n/G_{n+1}|< \infty,$$we say that $G$ is a bounded-order Vilenkin group. If $|G_n/G_{n+1}|$ is constant we say that $G$ is a constant-order Vilenkin group. 
\end{defi}
\begin{rem}\label{remmetric}
The above definition is linked to the following definition of ultrametric: for $x,y \in G$ their associated distance is defined as
\[\varrho_{\mathscr{G}} (x,y) = |x y^{-1}|_{\mathscr{G}} :=\begin{cases} 0 & \esp \esp \text{if} \esp x=y, \\ |G_n|  & \esp \esp \text{if} \esp x y^{-1} \in G_n \setminus G_{n+1}.\end{cases}\]In particular we write $|x|_\mathscr{G} := \varrho_{\mathscr{G}} (x,e).$ Note that this distance function relies on the selection of a particular sequence of compact open subgroups. However, for the purposes of this work, the choice of sequence is not significant. Thus we will use a fixed sequence of subgroups, resulting in a fixed ultrametric on the group.
\end{rem}

Fourier analysis on Vilenkin groups is still nowadays an active area of research, especially from the point of view of Vilenkin systems. In recent years G. Gát and U. Goginava studied the pointwise summability of Fourier-Vilenkin series \cite{Gt2020}, G. Tepnadze, L. Persson and coauthors explored in detail the properties of martingale Hardy spaces and Vilenkin systems \cite{Nadirashvili2022, Persson2022, Persson2017}, and M. Avdispahi, N. Memi, and F. Weisz investigated maximal functions, Hardy spaces and Fourier multiplier theorems on unbounded Vilenkin groups \cite{Avdispahi2012}. 

The problem of the convergence and decay rate of Fourier coefficients is a well-studied topic in Fourier analysis. For locally connected groups, especially Lie groups, one can relate the smoothness of a function to the decay rate of its Fourier coefficients. For instance, for a smooth function on the unit circle $f\in C^k(\T)$, we have $|\widehat{f}(j)|=o(|j|^{-k})$, and for $0<\alpha\leq 1$ and $f$ satisfying the H{\"o}lder condition of order $\alpha$, we have $|\widehat{f}(j)|=\mathcal{O}(|j|^{-\alpha})$. 

For Vilenkin groups the notion of differentiability is given in terms of the Vladimirov-Taibleson operator. To some extent, the theory of this operator is parallel to that of the fractional Laplacian on compact Lie groups, especially because it is nicely diagonalized by the group Fourier transform, just as the Laplace-Beltrami operator. Among other things, we can use Sobolev spaces $H^k_2(G)$ associated with the Vladimirov-Taibleson operator $D^k$ to study the decay of Fourier coefficients. We recall here the definition for the reader: $$D^k f(x) := C_k \int_G |y|_\mathscr{G}^{-k-1}(f(xy^{-1})-f(x))d \mu_G (y), \esp \esp k >0,$$ where $C_k$ is a normalization constant that will be specified later, and the Sobolev space $H^k_2(G)$ is defined as $$\{f\in L^2(\mu_G):D^k f\in L^2(\mu_G)\}.$$When $G=\Z_\ell$, or $G$ is any compact abelian Vilenkin group, we have once again, just like for the toroidal case, that $f \in H^k_2 (\Z_\ell)$ implies $|\widehat{f} (j)| = o (| j |_\ell^{-k})$, and for $0<\alpha\leq 1$ and $f$ satisfying the H{\"o}lder condition of order $\alpha$, we have $|\widehat{f}(j)|=\mathcal{O}(|j|_\ell^{-\alpha})$. See \cite{Younisthesis} for all the details.

Alternatively, we can investigate whether the Fourier transform belongs to some $L^p$-space. For example, the Hausdorff-Young inequality implies that $\widehat{f}\in L^q$ whenever $f\in L^p$ for $1<p\leq 2$ and $1/p+1/q=1$, and this fact can be strengthened if we assume some additional properties on $f$. Titchmarsh explored this idea in the context of Lipschitz spaces on $\R^n$, and Younis adapted his ideas to the setting of compact abelian groups, producing the first and second Titchmarsh theorems for Fourier series.

In this work, we continue in the same direction by considering functions $f\in Lip_\mathscr{G}(\alpha;p)$ in the Lipschitz space over a compact noncommutative Vilenkin group.

\begin{defi}\normalfont
Let $G$ be a compact Vilenkin group. Let $0 < \alpha \leq 1$ and $1 \leq p \leq \infty$ be real numbers. We define the Lipschitz space $Lip_\mathscr{G} (\alpha; p)$ by $$Lip_\mathscr{G} (\alpha; p) := \{ f \in L^p (\mu_G) \esp : \esp || f(h \cdot) - f(\cdot)||_{L^p (\mu_G)} = O (|h|_\mathscr{G}^\alpha) \esp \text{as} \esp |h|_\mathscr{G} \to 0\}.$$  
\end{defi}

For a compact group $G$, let $Rep(G)$ denote the collection of all continuous, finite dimensional, unitary irreducible representations of $G$. Let $\widehat{G}$ the collection of all irreducible representations in $Rep(G)$. As it is to be expected, the extension of the first and second Titchmarsh theorems to the noncommutative case require some non-trivial adjustments. This is best illustrated in \cite{DaherDelgadoRuzhansky} where, in order to obtain Titchmarsh theorems for compact homogeneous manifolds $G/K$, it is necessary to handle the behavior of the Fourier coefficients of the difference function $f(h \cdot) - f(\cdot)$ when $|h| \to 0$. These coefficients are $\widehat{f}(\xi) (\xi(h) - I_{d_\xi})$,  $[\xi] \in \widehat{G}$, so it all comes down to the behavior of the map $\xi(h) - I_{d_\xi}$ when $|h| \to 0$. 

In the profinite case the same issue needs to be addressed, but the situation is much more diverse. The unitary irreducible representations of a compact Vilenkin group $G$ have a non trivial kernel, which is a compact open subgroup, so $\xi(h) - I_{d_\xi}$ becomes identically zero when $|h|_\mathscr{G} \to 0$. And even when $\xi(h) - I_{d_\xi} \neq 0_{d_\xi}$, a very common situation is that $1$ is eigenvalue of $\xi(h)$, which is troublesome if one wants to estimate $\| \widehat{f}(\xi) (\xi(h) - I_{d_\xi}) \|_{HS}$ from below. Thus, in order to avoid this obstacle, we introduce the following condition on the representation theory of the group.

\begin{statement}{Condition \textbf{(A)}}
We say that a compact Vilenkin group $G$ satisfies Condition $(A)$ if there is a natural number $\mathscr{n}$ with the following property: for all $k\in \N_0$ there are $\mathscr{n}$ points $h_1,..., h_{\mathscr{n}}$ satisfying $|h_1|_{\mathscr{G}} = ... = |h_\mathscr{n}|_\mathscr{G} = |G/G_k|^{-1}$ and $$ |G/G_k|^q \langle \xi \rangle_{\mathscr{G}}^{-q}    
\| \widehat{f}(\xi)  \|_{HS}^q \lesssim \sum_{i=1}^{\mathscr{n}}\| \widehat{f}(\xi) (\xi(h_i) - I_{d_\xi}) \|_{HS}^q,  \esp \esp 1 \leq q < \infty,  $$for any unitary irreducible representation $[\xi]$ non trivial on $G_k \setminus G_{k+1}$. Here the function $\langle \cdot\rangle_\mathscr{G} : Rep(G) \to \R$ is defined in the following way: 
\[ \langle \pi \rangle_\mathscr{G} := \begin{cases}
1 \esp & \esp \text{if} \esp \esp \pi \esp \text{is the identity representation;} \\
|G/G_k| \esp & \esp \text{if} \esp \esp [\pi] \in Rep_k (G), \esp \esp n \in \N,  
\end{cases}
\]
\end{statement}
where $$G_k^\bot := \{ [\pi] \in Rep(G) \esp : \esp \pi |_{G_k} = I_{d_\pi}\}, \esp \esp Rep_k(G):= G^\bot_k \setminus G_{k-1}^\bot.$$ 
\begin{rem}
To understand the above condition it helps a lot to think on the abelian case. If $G= \Z_\ell$ then the unitary irreducible representations have the form $\chi_\xi (h):= e^{2 \pi i \{ h \xi \}_\ell}$, where $\xi \in \Q_\ell / \Z_\ell$ and $\{ \cdot \}_\ell$ denotes the $\ell$-adic fractional part function. Let $| h 
|_\ell = \ell^{-n}$ and $| \xi |_\ell = \ell^{m} $. When $m \leq n$, $ \chi_\xi (h) = 1$ and $\chi_\xi (h) -1 = 0$. When $m>n$, we can choose an $h$ such that $\chi_\xi (h)$ is a $\ell^{m - n}$ rooth of unity, so $$  |\chi_\xi (h) - 1 | \gtrsim \sin{
\pi \frac{a_h }{\ell^{m - n}}} \gtrsim \ell^{-(m - n)} = |\xi|_\ell^{-1} |h|_\ell^{-1} , \esp \esp 1 \leq a_h < \ell^{m - n}   ,$$which implies $|\xi|_\ell^{-1} |h|_\ell^{-1} |\widehat f(\xi)|  \lesssim |\chi_\xi (h) - 1 ||\widehat f(\xi)|$. In the non commutative case this kind lower bound is not trivial, thus the formulation of Condition (A). Given a noncommutative representation, it is possible in general for $\xi(h) $ to have $1$ as eigenvalue, so the estimate $|h|_\mathscr{G}^{-1} \langle \xi \rangle_{\mathscr{G}}^{-1}    
\| \widehat{f}(\xi)  \|_{HS}^q \lesssim \|  \widehat{f}(\xi) (\xi(h) - I_{d_\xi}) \|_{HS}^q$ might not be true for a single $h$. 

Condition (A) restricts a lot the class of compact Vilenkin groups for which our results apply, but it still includes many important examples. We will prove in Section 3 how any compact nilpotent Vilenkin group satisfies this condition, and in particular any compact nilpotent $\ell$-adic Lie group does aswell, where $\mathscr{n} = \mathscr{n}(\mathfrak{g})$ equals the index of nilpotency of the $\Z_\ell$-Lie algebra.                           
\end{rem}

Now with Condition (A) defined we are ready to state the main results of this paper. We proved for functions $f \in Lip_\mathscr{G}(\alpha ; p)$ the following two noncommutative versions of the Titchmarsh theorems:

\begin{teo}[First Titchmarsh Theorem]\label{firsttitchmarshtheorem}
Let $G$ be a compact Vilenkin group satisfying Condition (A). Let $p, q$ be real number such that $1< p \leq 2$ and $1/p + 1/q =1$. Let $0 < \alpha \leq 1$ and let $\alpha < \gamma< \alpha + \frac{1}{q}$. If $f \in Lip_\mathscr{G} (\alpha ; p)$ then $$\Big( \sum_{|G/G_k| < \langle \xi \rangle_\mathscr{G}} d_\xi^{q (\frac{2}{q} - \frac{1}{2})} ||\widehat{f}(\xi)||_{HS}^{q} \Big)^{1/q} = \mathcal{O}(|G/G_k|^{-\alpha}).$$Consequently $\widehat{D^{\gamma} f} \in L^{ \beta} (\widehat{G})$ for$$ \esp  \frac{q}{(\alpha - \gamma)q + 1} = \frac{p}{(\alpha - \gamma)p + p - 1} \leq \beta < q .$$Moreover $\widehat{f} \in L^\beta (\widehat{G})$ for $$  \frac{q}{ \alpha q + 1} = \frac{p}{\alpha p + p -1} < \beta \leq q.$$
\end{teo}
\begin{teo}[Second Titchmarsh theorem]\label{seconteo}
Let $0 < \alpha \leq 1$ and $f \in L^2 (\mu_G)$. Then $f \in Lip_\mathscr{G} (\alpha ; 2)$ if and only if $$\sum_{\langle \xi \rangle_{\mathscr{G}} > |G/G_k|} d_\xi ||\widehat{f}(\xi)||_{HS}^2 = \mathcal{O} (|G/G_k|^{-2 \alpha}) \esp \esp \text{as k} \to \infty.$$
\end{teo}

\begin{rem}
    Notice that in the hypothesis of the above theorems we did not assume that $G$ is a bounded-order Vilenkin group. 
\end{rem}

Something worth remarking about the above results is that, in contrast with \cite{DaherDelgadoRuzhansky}, it might seem like the dimension of the group does not appear in the above theorems. In fact, we would like to emphasize how this information is actually there, but it is hidden in the choice of ultrametric for the group. 

To do so, let $\ell >2$ be a prime number and let $G$ be a $d$-dimensional linear $\ell$-adic Lie sub-group of $GL_m (\Z_\ell )$, where $\Z_\ell$ denotes the compact ring of $\ell$-adic integers. We assume for the sake of our arguments that $G$ is nilpotent, so that it can be thought as the exponential image of a certain nilpotent $\Z_\ell$-Lie algebra $\mathfrak{g}$. By choosing a $\Z_\ell$-basis $\{X_1,...,X_d\}$ for $\mathfrak{g}$ we can write $$\mathfrak{g} = \Z_\ell X_1 + ...+ \Z_\ell X_d.$$I0=n this way we can identify points  $ \Z_\ell^d$ with points in the group via exponential map: $$(x_1,...,x_d):= \mathbb{exp} (x_1 X_1 + ...+  x_d X_d).$$
By endowing the group with the sequence of compact open subgroups $\mathscr{G}=\{G_n\}_{n \in \N_0}$, $G_n :=\mathbb{exp} ( \ell^n \mathfrak{g})$, we can see that $G$ is a compact Vilenkin group, and the associated ultrametric function $| \cdot |_\mathscr{G}$ satisfies $$|(x_1,...,x_d)|_\mathscr{G} = \| (x_1,...,x_d)\|^d_\ell, \esp \esp \esp \text{where} \esp \esp \| (x_1,...,x_d)\|_\ell := \max_{1 \leq j \leq d} | x_j |_\ell,$$which allow us to rewrite the Vladimirov-Taibleson operator in a more convenient form $$\mathscr{D}^\alpha f(x) := \frac{1 - \ell^\alpha}{1 - \ell^{- (\alpha + d)}} \int_G \frac{f(xy^{-1}) - f(x)}{\|y \|_\ell^{ \alpha + d}} d \mu_G(y).$$
The definition of the Lipchitz space over $G$ also changes slightly: 

\begin{defi}\normalfont
Let $G$ be a compact nilpotent linear $\ell$-adic Lie group. Let $0 < \alpha \leq 1$ and $1 \leq p \leq \infty$ be real numbers. We define the Lipschitz space $Lip_G (\alpha; p)$ by $$Lip_G (\alpha; p) := \{ f \in L^p (\mu_G ) \esp : \esp || f(h \cdot) - f(\cdot)||_{L^p (\mu_G)} = O (\|h \|_\ell^\alpha) \esp \text{as} \esp \|h\|_\ell \to 0\}.$$  
\end{defi}
Notice the differences between the definitions of $Lip_G (\alpha; p)$ and $Lip_\mathscr{G} (\alpha; p)$. Given a any metrizable profinite group $G$, the metric topology induces the existence of a sequence of compact open subgroups that coincide with the metric balls of the topology. And conversely, given a sequence of compact open subgroups we can always define an invariant ultrametric for the group as in Remark \ref{remmetric}. For any metrizable profinite group there are many choices of metric and sequence of subgroups, so in principle the Lipschitz space on $G$ depends on such choice, which is why we defined first the spaces $Lip_\mathscr{G} (\alpha; p)$. However, as we discussed previously, when $G$ is a $d$-dimensional compact $\ell$-adic Lie group there is a natural choice of ultrametric, i.e., the $\ell$-adic norm inherited from $\Z_\ell^d$. With this norm we can give a more "intrinsic" definition for the Lipchitz space independent of the choice of compact open subgroups, and in this case we use the notation $Lip_G (\alpha; p)$.

With the adapted version of the Lipchitz space over $G$ the Titchmarsh theorems take the following form. Notice how this time the dimension of the group appears in a similar way as in \cite{DaherDelgadoRuzhansky}.

\begin{teo}\label{firsttitchmarshtheoremLiegroups}
Let $G$ be a compact nilpotent $\ell$-adic Lie group of dimension $d$. Let $p, q$ be real number such that $1< p \leq 2$ and $1/p + 1/q =1$. Let $0 < \alpha \leq 1
$ and let $\alpha < \gamma< \alpha + \frac{d}{q}$. If $f \in Lip_G (\alpha ; p)$ then $$\sum_{\ell^{-k} < \langle \xi \rangle_G  } d_\xi^{q (\frac{2}{q} - \frac{1}{2})} || \widehat{f} (\xi)||_{HS}^{q} = \mathcal{O} (\ell^{ -k \alpha q}) .$$Consequently $\widehat{\mathscr{D}^{\gamma} f} \in L^{ \beta} (\widehat{G})$ for$$ \esp  \frac{q d}{(\alpha - \gamma)q + d} = \frac{ d}{(\alpha - \gamma) + d - \frac{d}{p}} \leq \beta \leq q .$$Moreover $\widehat{f} \in L^\beta (\widehat{G})$ for $$  \frac{q d}{ \alpha q + d} = \frac{p d}{\alpha p + p d -d} < \beta \leq q.$$ Here $\langle \xi \rangle_G$ denotes the eigenvalue of the operator $$\mathbb{D}:=    \big( \frac{1- \ell^{-d}}{1-\ell^{-(1 + d)}} I+ \mathscr{D}^1 \big),$$corresponding to the matrix entries of the representation $[\xi] \in \widehat{G}$.
\end{teo}
\begin{teo}\label{seconteoLiegroups}
Let $0 < \alpha \leq 1$ and $f \in L^2 (\mu_G)$. Then $f \in Lip_G (\alpha ; 2)$ if and only if $$\sum_{\langle \xi \rangle_G > \ell^k} d_\xi ||\widehat{f}(\xi)||_{HS}^2 = \mathcal{O} (\ell^{-k2 \alpha}) \esp \esp \text{as k} \to \infty.$$
\end{teo}

\begin{rem}\label{remheisenberggroup}
We will prove in Section 3 how compact nilpotent Vilenkin groups satisfy Condition (A). In order to introduce the ideas of the proof, and to illustrate better the meaning of Condition (A), let us consider the Heisenberg group $\mathbb{H}_d $ over $\Z_\ell$, in turn defined as:  

\[
\mathbb{H}_{d}(\Z_\ell)= \left\{
  \begin{bmatrix}
    1 & x^t & z \\
    0 & I_{m} & y \\
    0 & 0 & 1 
  \end{bmatrix}\in GL_{d+2}(\Z_\ell) \esp : \esp x , y \in \Z_\ell^d, \esp z \in \Z_\ell \right\}. 
\]

Any unitary irreducible representation $[\pi]$ of $\mathbb{H}_d$ has two possibilities: either it is trivial on the center $\mathcal{Z}(\mathbb{H}_d)$, or it is not and it is the induced representation of a central character. If $[\pi]$ is trivial in the center then it descents to a representation of the abelian group $\mathbb{H}_d/ \mathcal{Z}(\mathbb{H}_d) \cong \Z_\ell^{2d}$, so that $\pi$ is of the form $$\pi_{\xi , \eta} (x,y,z) := e^{2 \pi i \{x \xi + y \eta \}_\ell}, \esp \esp \esp \xi, \eta \in (\Q_\ell / \Z_\ell )^{d}.$$If $[\pi]$ is non-trivial in the center, then the restriction of the representation to the center gives rise to a central character $e^{2 \pi i \{z \lambda \}_\ell}$, and the representation $[ \pi]$ is the induced representation by said character. Hence, the noncommutative representations of $\mathbb{H}_d$ are indexed by $\lambda \in \Q_\ell / \Z_\ell$ and they satisfy  $$\pi_\lambda  (0,0,z)  = e^{2 \pi i \{z \lambda \}_\ell} I_{d_\lambda}.$$Let us denote by $\widehat{\mathbb{H}}_d$ the collection of all unitary irreducible representations of $\mathbb{H}_d$. Then we can write $$\widehat{\mathbb{H}}_d= \{ [\pi_{(\xi , \eta, \lambda)}] =  [\pi_{\xi , \eta}] \otimes  [\pi_{\lambda}] \esp : \esp \xi, \eta \in (\Q_\ell / \Z_\ell)^d , \esp \lambda \in \Q_\ell / \Z_\ell \}.$$ 
Now we claim that $\mathbb{H}_d$ satisfies Condition (A) with $\mathscr{n} = 2$. To see it, just fix $n \in \N_0$. Then we can always take $h_1 := (x_0 ,y_0, 0)$ and $h_2 := (0,0, z_0)$ where $$\|x_0 \|_\ell = \|y_0 \|_\ell = |z_0|_\ell = \ell^{-n},$$so that for any representation $\pi_{(\xi , \eta, \lambda)}$ with $\| (\xi ,\eta,  \lambda \|_\ell > \ell^n $ it holds: 

$$\pi_{(\xi , \eta, \lambda)} (x_0 , y_0 , 0) = e^{2 \pi i \{x_0 \xi + y_0 \eta \}_\ell} \pi_\lambda (x_0 , y_0 , 0), \esp \esp \text{and} \esp \esp \pi_{(\xi , \eta, \lambda)} (0 , 0 , z_0) = e^{2 \pi i \{z_0 \lambda \}_\ell}  I_{d_\lambda}.$$If $|\lambda|_\ell=1$, then $$ |\pi_{(\xi , \eta, \lambda)} (x_0 , y_0 , 0) - 1| \gtrsim \ell^{n} \| (\xi, \eta) \|_{\ell}^{-1} = \ell^{n} \|(\xi , \eta , \lambda) \|_\ell^{-1} .$$If $|\lambda|_\ell > 1$, $$ \| \widehat{f}(\xi , \eta , \lambda) ( \pi_{(\xi , \eta, \lambda)} (0 , 0 , z_0) - I_{d_\lambda}) \|_{HS} \gtrsim \ell^{n} |\lambda|_\ell^{-1} \| \widehat{f}(\xi , \eta , \lambda) \|_{HS} \geq \ell^{n} \|(\xi , \eta , \lambda) \|_\ell^{-1} \| \widehat{f}(\xi , \eta , \lambda) \|_{HS},$$for any $\widehat{f}(\xi , \eta , \lambda) \in \C^{d_\lambda \times d_\lambda}$. This implies that $\mathbb{H}_d (\Z_\ell)$ satisfies Condition (A) since we can easily check that $\langle \pi_{(\xi , \eta , \lambda)} \rangle_G \asymp \|(\xi , \eta , \lambda) \|_\ell$.    
\end{rem}

    An alternative approach to classifying functions on a compact Vilenkin group $G$ in terms of their H{\"o}lder-Lipschitz continuity, is to study their modulus of continuity. This is Platonov's approach in \cite{Platonov1, Platonov2} for locally compact abelian Vilenkin groups and, as we will show in this paper, the same techniques are applicable to compact Vilenkin groups, after adjusting some details related to the representation theory of the group. 

\begin{defi}\label{defimodofcont}\normalfont
\esp
    \begin{enumerate}
        \item[(i)] Let $f \in L^2(\mu_G)$. We define the \emph{modulus of continuity} associated to $f$ as the monotonous decreasing sequence of non-negative real numbers $\omega_2 (f, \mathscr{G}):=\{ \omega_2 (f, \mathscr{G}, n) \}_{n \in \N_0}$ given by: $$\omega_2 (f, \mathscr{G}, n):= \sup_{h \in G_n } \|f(h \cdot ) - f(\cdot) \|_{L^2 (\mu_G)}.$$
        \item[(ii)] Let $\omega:=\{ \omega_n \}_{n \in \N_0}$  be a sequence of real numbers monotonously decreasing to zero. We say that a function $f$ belongs to the space $\mathscr{M}^\omega_2 (\mathscr{G})$ if $f \in L^2 (\mu_G)$ and for some constant $C_f >0$ it holds: $$\omega_2 (f,\mathscr{G}, n) \leq C_f \omega_n, \esp \esp n \in \N_0.$$  
    \end{enumerate}
\end{defi}

\begin{rem}
    Let $\omega:=\{ \omega_n \}_{n \in \N_0}$ and $\omega':=\{ \omega_n' \}_{n \in \N_0}$ be sequences of real numbers monotonously decreasing to zero. If $\omega$ and $\omega'$ are equivalent, that is, if there exists $C_1 , C_2 >0$ such that $$C_1 \omega_n' \leq \omega_n \leq C_2 \omega_n' , \esp \esp \text{for all} \esp n \in \N_0,$$then the associated spaces $\mathscr{M}^\omega_2 (\mathscr{G})$ and $\mathscr{M}^{\omega'}_2 (\mathscr{G})$ coincide.
\end{rem}

The notion of modulus of continuity introduced in Definition \ref{defimodofcont} may be used to generalize the notion of H{\"o}lder-Lipchitz continuity. That is Platonov's idea in \cite{Platonov1, Platonov2} and we are going to exploit the same argument here. The only thing we need to notice is how the modulus of continuity of a function $f \in L^2 (\mu_G)$ bounds the tails of the $L^2$-norm of the Fourier transform, as it is established in Theorem \ref{ProPlatonov1}. 

\begin{teo}\label{ProPlatonov1}
Let $G$ be a compact Vilenkin group satisfying Condition (A). For every $f \in L^2 (\mu_G)$ the following inequalities hold:
$$\frac{1}{2} \omega_2( f , \mathscr{G}, n) \leq \sum_{\langle \xi \rangle_{\mathscr{G}} > |G/G_k|} d_\xi ||\widehat{f}(\xi)||_{HS}^2 \leq \frac{1}{\sqrt{2}} \omega_2( f , \mathscr{G}, n), $$ where the constant $1/\sqrt{2}$ appearing in the right hand inequality is sharp. Consequently $f \in \mathscr{M}_2^{\omega}(\mathscr{G})$ if and only if $$\Big(\sum_{\langle \xi \rangle_{\mathscr{G}} > |G/G_k|} d_\xi ||\widehat{f}(\xi)||_{HS}^2 \Big)^{1/2} \leq C \omega_n ,$$for some $C>0.$ 
\end{teo}

The above theorem, together with Theorem \ref{seconteo}, produce the conclusion $\mathscr{M}^{\omega (\alpha)}_2 (\mathscr{G}) = Lip_{\mathscr{G}} (\alpha ; 2)$ when $\omega (\alpha)$ is the sequence given by $\omega_n(\alpha):= |G/G_n |^{-\alpha}$. In other words, the spaces $Lip_\mathscr{G} (\alpha ; 2)$ are particular cases of the functional spaces $\mathscr{M}_2^{\omega}(G)$ defined in Definition \ref{defimodofcont}, and we can produce many more examples of interesting functional spaces by choosing different sequences $\omega$. For example, we can choose the sequence $\omega^{DL}(\alpha , \nu)$ defined by $$\omega^{DL}_n (\alpha , \nu) : = |G/G_n|^{- \alpha} (\log{|G/G_n|})^\nu,$$and argue as before that the space $\mathscr{M}^{\omega^{DL}(\alpha ,\nu)}_2 (\mathscr{G})$ coincides the space $DL_\mathscr{G}(\alpha, \nu ; p)$ for $p=2.$ Here $DL_\mathscr{G}(\alpha, \nu ; p)$ denotes the space of Dini-Lipschitz functions on $G$ which is defined as follows: 

\begin{defi}\label{DefiDiniLip}
    Let $G$ be a compact Vilenkin group together with the sequence of compact open subgroups $\mathscr{G}:= \{G_n \}_{n \in \N_0}$. Let $0 < \alpha \leq 1$, $\nu \in \R$, $1 < p \leq 2$, and let $q$ be such that $\frac{1}{p} + \frac{1}{q} = 1$. We say that $f \in L^p (\mu_G) $ belongs to the space $DL_\mathscr{G}(\alpha, \nu ; p)$ if $$\| f(h \cdot ) - f(\cdot) \|_{L^p(\mu_G)} = \mathcal{O} \Big( |h|_\mathscr{G}^\alpha \big(\log{\frac{1}{|h|_\mathscr{G}}}\big)^{\nu} \Big), \esp \esp \text{as} \esp \esp |h|_{\mathscr{G}} \to 0.$$
\end{defi}

For general $p \in (1,2]$ we obtain the analog of Theorem \ref{firsttitchmarshtheorem} for Dini-Lipschitz functions:  

\begin{teo}\label{toeDL1}
Let $G$ be a compact Vilenkin group satisfying Condition (A). Let $p,q$ be real numbers such that $1< p \leq 2,$ and $\frac{1}{p} + \frac{1}{q} = 1$. Let $0< \alpha \leq 1$, $\nu \in \R$ and $\alpha < \gamma < \alpha +1/q$. If $f \in DL_\mathscr{G}(\alpha, \nu ; p)$,  then we have $$\Big( \sum_{|G/G_k| < \langle \xi \rangle_\mathscr{G}} d_\xi^{q (\frac{2}{q} - \frac{1}{2})} ||\widehat{f}(\xi)||_{HS}^{q} \Big)^{1/q} = \mathcal{O}(|G/G_k|^{-\alpha} ( \log{|G/G_k|} )^\nu).$$Consequently $\widehat{D^{\gamma} f} \in L^{ \beta} (\widehat{G})$ provided either $$ \esp  \frac{q}{(\alpha - \gamma)q + 1} = \frac{p}{(\alpha - \gamma)p + p - 1} \leq \beta \leq q , \esp \esp \text{and} \esp \esp \nu = 0,$$ or$$ \esp  \frac{q}{(\alpha - \gamma)q + 1} = \frac{p}{(\alpha - \gamma)p + p - 1} < \beta \leq q , \esp \esp \text{and} \esp \esp \nu \in \R.$$ Moreover $\widehat{f} \in L^\beta (\widehat{G})$ for $$  \frac{q}{ \alpha q + 1} = \frac{p}{\alpha p + p -1} < \beta \leq q.$$
\end{teo}
And for $p=2$ we can prove that the space $DL_\mathscr{G}(\alpha, \beta ; 2)$ coincides with the space $\mathscr{M}_2^{\omega^{DL} (\alpha , \beta)}(\mathscr{G}).$
\begin{teo}\label{teoDL2}
    Let $\alpha \geq 0$ and $\nu \in \R$. Then the conditions $$\| f(\cdot h) - f(\cdot) \|_{L^2(\mu_G)} = \mathcal{O} \Big( |h|_\mathscr{G}^\alpha \big(\log{\frac{1}{|h|_\mathscr{G}}}\big)^{\nu} \Big),$$and $$\sum_{\langle \xi \rangle_{\mathscr{G}} > |G/G_k|} d_\xi ||\widehat{f}(\xi)||_{HS}^2 = O (|G/G_k|^{-2 \alpha} (\log{|G/G_k|})^{2\nu}) \esp \esp \text{as k} \to \infty,$$ are equivalent. Moreover, the space $DL_\mathscr{G}(\alpha, \nu ; 2)$ coincides with the space $\mathscr{M}_2^{\omega^{DL} (\alpha , \nu)}(\mathscr{G}).$
\end{teo}

We will organize our exposition as follows: 

\begin{itemize}
    \item In Section 2 we recall the basics on the Fourier analysis on compact groups. In particular, we discuss briefly how the representations of profinite groups are special, and we use their properties to calculate the symbol of the Vladimirov-Taibleson operator. 
    \item Section 3 is dedicated to the proof of the Titchmarsh theorems for H{\"o}lder-Lipschitz functions in the context of compact non-commutative Vilenkin groups, under the assumption of Condition (A). We also prove how compact nilpotent Vilenkin groups satisfy Condition (A).
    \item In Section 4 we adjust the necessary details to adapt the Titchmarsh theorems to compact $\ell$-adic Lie groups. In the process, we discuss how our conclusions differ slightly from the conclusions in \cite{DaherDelgadoRuzhansky}, and why we think our versions of the Titchmarsh theorems for H{\"o}lder-Lipschitz functions are an improvement. 
    \item In Section 5 we provide the proof for Theorem \ref{ProPlatonov1}.
    \item In Section 6 we prove the Dini-Lipschitz versions of the first and second Titchmarsh theorem. 
\end{itemize}

\section{Preliminaries}
In this section, we recall some basic facts about the Fourier analysis on compact noncommutative groups, and some basic properties of the Vladimirov-Taibleson operator in this context. For a detailed exposition on Fourier analysis on compact groups see \cite{Ruzhansky2010}. For a recent exposition on the Vladimirov-Taibleson operator see \cite{2022arXiv220907998K} and the references therein.  
\subsection{Notation}
Let $G$ be a compact Vilenkin group together with the sequence of compact open subgroups $\{G_n\}_{n \in \N_0}$. We will denote by $Rep(G)$ the collection of all continuous finite-dimensional unitary representations on $G$. $\widehat{G}$ will denote the collection of irreducible representation in $Rep(G)$. 

Let $\mu_G$ the unique positive normalised translation-invariant Haar measure on $G$. By the Peter-Weyl Theorem the matrix entries of the representations $[\xi] \in \widehat{G}$ form an orthonormal basis of the space $L^2(G)$. Therefore any function $f \in L^2 (G)$ can be written as $$f(x) =\mathcal{F}_G^{-1} \circ \mathcal{F}_G [f](x) = \mathcal{F}_G^{-1}[\widehat{f}](x)= \sum_{[\xi] \in \widehat{G} } d_\xi Tr(\xi(x) \widehat{f} (\xi)).$$In this case $\widehat{f}(\xi)$ denotes the Fourier coefficient of $f$ with respect to $[\xi] \in \widehat{G}$,  defined as $$\widehat{f}(\xi) = \mathcal{F}_G [f] (\xi) := \int_G f(x) \xi^* (x) d \mu_G (x), $$ and the inverse Fourier transform takes the form: $$\mathcal{F}_G^{-1} \varphi (x) := \sum_{[\xi] \in \widehat{G}} d_\xi Tr[\xi(x) \varphi (\xi)].$$

Another consequence of the Peter-Weyl theorem is the Plancherel identity, which for compact groups has the form: 

$$\| f \|_{L^2(\mu_G)}^2 = \sum_{[\xi] \in \widehat{G}} d_\xi \| \widehat{f} (\xi) \|_{HS}^2 =: \| \widehat{f}(\xi) \|_{L^2(\widehat{G})}^2.$$If we define the space $L^\infty(\widehat{G})$ as $$L^\infty(\widehat{G}) := \{ \varphi : \widehat{G} \to  \bigcup_{[\xi] \in \widehat{G}} \C^{d_\xi \times d_\xi} \esp : \esp  \varphi (\xi) \in \C^{d_\xi \times  d_\xi} \esp , \esp \esp \text{and} \esp \esp 
\sup_{[\xi] \in \widehat{G}} d_\xi^{-1/2} \| \varphi (\xi) \|_{HS} < \infty  \},$$the usual Stein-Weiss interpolation between weighted spaces produces the Hausdorff-Young inequality for compact groups:

\begin{teo}[Hausdorff-Young inequality]
    Let $1 \leq p \leq 2$ and $\frac{1}{p} + \frac{1}{q} = 1$. Let $f \in L^p (\mu_G)$ and $\varphi \in L^q (\widehat{G})$. Then $\| \widehat{f} \|_{L^q(\widehat{G})} \leq \| f \|_{L^p (\mu_G)}$ and $\| \varphi \|_{L^q(\widehat{G})} \leq \| \mathcal{F}_G^{-1} \varphi \|_{L^p(\mu_G)}$. Here for $1 \leq q < \infty$ the $L^q$-space over $\widehat{G}$ is defined as: 
    $$L^q(\widehat{G}) := \{ \varphi : \widehat{G} \to  \bigcup_{[\xi] \in \widehat{G}} \C^{d_\xi \times d_\xi} \esp : \esp  \varphi (\xi) \in \C^{d_\xi \times  d_\xi} \esp , \esp \esp \text{and} \esp \esp 
\sum_{[\xi] \in \widehat{G}} d_\xi^{q(\frac{2}{q} - \frac{1}{2})} \| \widehat{f} (\xi) \|_{HS}^q < \infty  \},$$and $$\| \varphi  \|_{L^q (\widehat{G})}= \Big( \sum_{[\xi] \in \widehat{G}} d_\xi^{q(\frac{2}{q} - \frac{1}{2})} \| \widehat{f} (\xi) \|_{HS}^q\Big)^{1/q}.$$
\end{teo}
 
For every equivalent class $[\pi] \in Rep(G)$ the matrix representation $\pi:G \to GL_{d_\pi} (\C)$ has a non trivial kernel $K_\pi$. This kernel is compact and open in $G$, and it must contain one of the subgroups $G_n$. Let $n_\pi := \min \{n \in \N_0 : \pi |_{G_n} = I_{d_\pi} \}.$ We also use the notation 
$$G_n^\bot := \{ [\pi] \in Rep(G) \esp : \esp \pi |_{G_n} = I_{d_\pi}\}, \esp \esp Rep_n(G):= G^\bot_n \setminus G_{n-1}^\bot, \esp \esp \widehat{G}_n := Rep_n (G) \cap \widehat{G}.$$ 
This gives rise to the decomposition of $Rep(G)$ and $\widehat{G}$ as the disjoint unions $$Rep(G)=\bigcup_{n \in \N_0} Rep_n(G) , \esp \esp \text{and} \esp \esp \widehat{G} =\bigcup_{n \in \N_0} \widehat{G}_n,$$and in this way any $f \in L^2(G)$ can be written as $$f(x)= \sum_{n \in \N_0} \sum_{[\xi] \in \widehat{G}_n} d_\xi Tr(\xi(x) \widehat{f}(\xi)).$$The Fourier series representation of functions in $L^2(G)$ implies that the action of a densely defined linear operator $T$ on any function $f \in L^2 (G)$ can be written as $$T f (x) = \sum_{n \in \N_0} \sum_{[\xi] \in \widehat{G}_n} d_\xi Tr \big( \xi (x) \sigma_T (x , \xi) \widehat{f} (\xi) \big).$$Here the symbol of the operator is defined as $$\sigma_T (x , \xi):=\xi^* (x) T \xi (x),$$and in particular it does not depend on the variable $x \in G$ when $T$ is left invariant.  

\subsection{The spectrum of the Vladimirov-Taibleson operator}

In this subsection we introduce our most important example of left-invariant operator on $G$. We call it the \emph{Vladimirov-Taibleson operator} and in our analysis, it plays a very similar role to the laplacian on compact Lie groups. To define it formally, it is convenient to introduce first some convenient notation.

\begin{defi}\label{weigth}\normalfont

We define the function $\langle \cdot\rangle_\mathscr{G} : Rep(G) \to \R$ in the following way: 
\[ \langle \pi \rangle_\mathscr{G} := \begin{cases}
1 \esp & \esp \text{if} \esp \esp \pi \esp \text{is the identity representation;} \\
|G/G_n| \esp & \esp \text{if} \esp \esp [\pi] \in Rep_n (G), \esp \esp n \in \N.  
\end{cases}
\]
\end{defi}
\begin{defi}\label{defigroupgammafunction}\normalfont
Let $G$ be a compact Vilenkin group and let $a>0$. We will use the notation $$\Gamma_{\mathscr{G}} (a , n):= \Big( \frac{1}{|G_{n-1}/G_n|^{a + 1}} +\sum_{k=0}^ {n-1}\frac{1}{|G_k/G_n|^{a }} \big(1- \frac{1}{|G_k/G_{k+1}|}\big) \Big).$$We will call the number $$\Gamma_{\mathscr{G}} (a) := \sup_{n \in \N_0} \Gamma_{\mathscr{G}}(a , n),$$ the group Gamma function evaluated in $a$, whenever this number exists. In particular, if the sequence of subgroups $\mathscr{G}$ is constant-ordered, that is, $|G_{n} / G_{n+1}| = \kappa$ for all $n \in \N_0$, then $$\Gamma_{\mathscr{G}}(a) = - \frac{1- \kappa^{-(a+1)}}{1 - \kappa^{a}}. $$ 
\end{defi}
The purpose of defining the above function is to use it as a normalization constant in the definition of the Vladimirov-Taibleson operator: 

\begin{defi}\label{defivladimirovoperator}\normalfont
Let $G$ be a compact Vilenkin group with a sequence of compact open subgroups $\mathscr{G}$. For $a >0$ we define on $\mathcal{D} (G)$ the \emph{Vladimirov-Taibleson operator} on $G$ by the formula 
$$D^a f (x) :=  \frac{-1}{\Gamma_{\mathscr{G}}(a)} \int_G \frac{f (xy^{-1}) - f(x)}{| y|_{\mathscr{G}}^{a + 1} } dy.$$
\end{defi}
To calculate the spectrum of this operator we just need to notice that it is a left invariant operator, so it is diagonalized by the Fourier transform. Because of that we just have to check out its associated symbol, which is by definition
$$\sigma_{D^a} (\pi) := \pi^* (x) D^a \pi (x) = \frac{-1}{\Gamma_{\mathscr{G}}(a)}\int_G \frac{\pi (z) - I_{d_\xi} }{|z|_{\mathscr{G}}^{a+1}} dz, \esp \esp \pi \in Rep(G).  $$Let us work a bit on this integral. Assume that $[\pi] \in Rep_n (G)$. Then \begin{align*}
    \int_{G} \frac{\pi (z) - I_{d_\pi} }{|z|_{\mathscr{G}}^{a+1}} dz &= \int_{G \setminus G_n} \frac{\pi (z) - I_{d_\pi} }{|z|_{\mathscr{G}}^{a+1}} dz,
\end{align*} and we can divide the integral in two parts. In one side \begin{align*}
    \int_{G \setminus G_n} \frac{I_{d_\pi} }{|z|_{\mathscr{G}}^{a+1}} dz &= \sum_{k=0}^{n-1} | G_k |^{-(a +1)} \int_{G_k \setminus G_{k+1}} I_{d_\pi}  dz \\&= \sum_{k=0}^{n-1} | G_k|^{-(a +1) } \big(|G_k| - |G_{k+1}| \big) I_{d_\pi} \\   &= \sum_{k=0}^{n-1} | G/G_k|^{{a +1} } |G/G_k|^{-1} \big(1 - |G_k/G_{k+1}|^{-1} \big) I_{d_\pi}\\ &=| G/G_n|^{a  } \sum_{k=0}^{n-1}\frac{1}{|G_k/G_n|^{a }} \big(1- \frac{1}{|G_k/G_{k+1}|} \big) I_{d_\pi}.
\end{align*}For the remaining integral we use the following well-known property of the representations of a profinite group:
\begin{pro}\label{proIntegralRep}
    Let $G$ be a compact Vilenkin group and let $[\xi] \in  \widehat{G}$. Then 
\[ \int_{G_n} \xi (x) dx = 
\begin{cases}
     |G_n| I_{d_\xi}, \esp & \esp \text{if} \esp \esp \langle \xi \rangle_{\mathscr{G}} \leq |G/G_n| \\
     0_{d_\xi}, \esp & \esp \text{if} \esp \esp \langle \xi \rangle_{\mathscr{G}} > |G/G_n|.
     
\end{cases}
\]
\end{pro}By applying Proposition \ref{proIntegralRep} we get: 
\begin{align*}
    \int_{G \setminus G_n } \frac{\pi (z)}{|z|_G^{a + 1}}dz &= \sum_{k=0}^{n-1}  |G/G_k|^{a +1}  \int_{G_k \setminus G_{k+1}} \pi (x) d \mu (x),\\ &= \frac{-|G/G_{n-1}|^{a+1}}{|G/G_n|} I_{d_\pi} \\ &= -|G/G_{n}|^{a} \frac{1}{|G_{n-1}/G_n|^{a + 1}} I_{d_\pi}, 
\end{align*}so the symbol of the Vladimirov-Taibleson operator is $$\sigma_{D^a} (\pi) = \frac{1}{\Gamma_{\mathscr{G}}(a)} \Big( \frac{1}{|G_{n-1}/G_n|^{a + 1}} +\sum_{k=0}^{n-1}\frac{1}{|G_k/G_n|^{a }} \big(1- \frac{1}{|G_k/G_{k+1}|}\big) \Big) |G/G_n|^a   I_{d_\pi},$$which justifies Definition \ref{defigroupgammafunction}. 
Now for $[\xi] \in \widehat{G}$ we can write: $$\sigma_{D^a} (\xi) = \frac{1}{\Gamma_{\mathscr{G}}(a)} \Big( \frac{1}{|G_{n-1}/G_n|^{a + 1}} +\sum_{k=0}^{n-1}\frac{1}{|G_k/G_n|^{a }} \big(1- \frac{1}{|G_k/G_{k+1}|}\big) \Big)  |G/G_n|^{a} I_{d_\xi} =: \lambda_\xi (D^a) I_{d_\xi},$$ where $$\lambda_\xi ({D}^{a}) = \frac{\Gamma_{\mathscr{G}} (a , n )}{\Gamma_{\mathscr{G}}(a)} |G/G_n|^{a}\asymp |G/G_n|^a,$$ for $[\xi] \in \widehat{G}_n.$ That is, the Fourier transform diagonalize the Vladimirov-Taibleson operator, and its associated eigenvalues $\lambda_\xi (D^a)$ satisfy $\lambda_\xi ({D}^{a}) \asymp \langle \xi \rangle_\mathscr{G}^a$.

In the special case where $G$ is a constant order Vilenkin group, let us say $|G_n/G_{n+1}|= \varkappa$ for all $n \in \N_0$, the Vladimirov-Taibleson operator is going to be $$D^a f(x) :=\frac{1 - \varkappa^a}{1 - \varkappa^{-(a + 1)}}\int_G \frac{f(x+y) - f(x)}{|y|_{\mathscr{G}}^{a +1}} dy , $$and its associated symbol would be \begin{align*}
    \sigma_{D^a } (\xi) &= - \frac{1 - \varkappa^a}{1 - \varkappa^{-(a + 1)}} \varkappa^{a n}  \Big( \varkappa^{-(a +1)} +(1 - \varkappa^{-1})\varkappa^{-a}\frac{1 - \varkappa^{-a n}}{1 - \varkappa^{-a}} \Big)I_{d_\xi} \\ &=   \varkappa^{a n} I_{d_\xi} + \varkappa^{- a} \frac{1 - \varkappa^{-1}}{1 - \varkappa^{- a}} \frac{1 - \varkappa^a}{1 - \varkappa^{-(a + 1)}}I_{d_\xi} \\ &= \varkappa^{a n} I_{d_\xi} - \frac{1 - \varkappa^{-1}}{1 - \varkappa^{ -(a+1)}} I_{d_\xi},\esp \esp \esp \text{for} \esp [\xi] \in \widehat{G}_n.
\end{align*}
In particular, when $G$ is a nilpotent compact $\ell$-adic Lie group, the symbol of the operator $$\mathscr{D}^\alpha f(x) := \frac{1 - \ell^\alpha}{1 - \ell^{- (\alpha + d)}} \int_G \frac{f(xy^{-1}) - f(x)}{\|y \|_\ell^{ \alpha + d}} dy,$$is going to be 

$$\sigma_{\mathscr{D}^\alpha} (\xi) =\ell^{\alpha n} I_{d_\xi} - \frac{1 - \ell^{-d}}{1 - \ell^{ -(\alpha+d)}} I_{d_\xi},\esp \esp \esp \text{for} \esp [\xi] \in \widehat{G}_n.$$In this case we can interpret the bracket $\langle \xi \rangle_G$ appearing in Theorems \ref{firsttitchmarshtheoremLiegroups} and \ref{seconteoLiegroups} as the eigenvalues of the operator $$\mathbb{D} :=    \big( \frac{1- \ell^{-d}}{1-\ell^{-(1 + d)}} I+ \mathscr{D}^1 \big),$$ which are given by:

\[ \langle \xi \rangle_G := \lambda_\xi (\mathbb{D})= \begin{cases}
 \frac{1- \ell^{-d}}{1-\ell^{-(1 + d)}}  \esp & \esp \text{if} \esp \esp [\xi] \esp \text{is the identity representation;} \\
p^{n} \esp & \esp \text{if} \esp \esp [\xi] \in \widehat{G}_n , \esp \esp n \in \N.  
\end{cases}
\]

\section{Proofs for Vilenkin groups}
This section is dedicated to the proof of the Titchmarsh theorems for H{\"o}lder-Lipschitz functions on compact Vilenkin groups. We will start by showing how compact nilpotent Vilenkin groups satisfy Condition (A), and how to construct more examples.  

\begin{lema}\label{lemazerorep}
Let $G$ be a compact nilpotent Vilenkin group. Let $\mathscr{n} = \mathscr{n}(G)$ be the nilpotency class of $G$. Then for all $k\in \N_0$ there are $\mathscr{n}$ points $h_1,..., h_{\mathscr{n}} \in G$ satisfying $|h_1|_{\mathscr{G}} = ... = |h_\mathscr{n}|_\mathscr{G} = |G/G_k|^{-1}$ and 
    $$ |G/G_k|^q \langle \pi \rangle_{\mathscr{G}}^{-q}    
\| \widehat{f}(\pi)  \|_{HS}^q \lesssim \sum_{i=1}^{\mathscr{n}}\| \widehat{f}(\pi) (\pi(h_i) - I_{d_\pi}) \|_{HS}^q,  \esp \esp 1 \leq q < \infty,  $$for every unitary irreducible representation $[\pi] \in \widehat{G}$ non trivial on $G_k$, and any $ \widehat{f}(\pi) \in \C^{d_\pi \times d_\pi}$.
\end{lema}

\begin{proof}
We follow the same arguments as in Remark \ref{remheisenberggroup} for the Heisenberg group $\mathbb{H}_d (\Z_\ell)$. Let us proceed by induction on the nilpotency class $\mathscr{n} = \mathscr{n}(G)$ of $G$: 

\begin{itemize}
    \item If $\mathscr{n} = 2$ we have something similar to the case of the Heisenberg group. There must be two kinds of representations: the characters $\pi_\xi =\chi_\xi$ of the abelian group $G/\mathcal{Z}(G)$, and the noncommutative representations $\pi_\lambda$ induced by a central character $\chi_\lambda.$ All the representations of $G$ are of the form $\pi_{(\xi , \lambda)} = \chi_\xi \otimes\pi_\lambda $ so, we can take $h_1 \in G/\mathcal{Z}(G) $, and $h_2 \in \mathcal{Z}(G)$, with $|h_1|_\mathscr{G} = |h_2|_\mathscr{G} = |G/G_n|^{-1}$. In this way, for $\langle \pi_{(\xi , \lambda)} \rangle_{\mathscr{G}} = |G/ G_m|$, the matrices $\pi_{(\xi , \lambda)} (h_i)$, $i=1,2$, are not only unitary but it is also true that $(\pi_{(\xi , \lambda)}(h_i))^{|G_n / G_m|} = I_{d_{\lambda}}$. For the characters $\chi_\xi$ this means that $\chi_\xi (h_1)$ is a $|G_n / G_m|$-rooth of the unity, and thus $$ |G/G_n| |G/G_m|^{-1}= |G_n / G_m|^{-1} \lesssim |\chi_\xi (h_1) - 1 |.$$ Similarly, for the noncommutative representations we have $\pi_{(\xi , \lambda)}(h_2) - I_{d_\lambda} = (\chi_\lambda (h_2) - 1) I_{d_\lambda}$ so that $$|G/G_n| |G/G_m|^{-1}  \| \widehat{f} (\pi_{(\xi , \lambda)}) \|_{HS} \lesssim \| \widehat{f} (\pi_{(\xi , \lambda)})(\pi_{(\xi , \lambda)}(h_2) - I_{d_\lambda}) \|.$$In conclusion we arrive to $$ |G/G_n|^q \langle \pi_{(\xi , \lambda)} \rangle_{\mathscr{G}}^{-q}    
    \| \widehat{f}(\pi_{(\xi , \lambda)})  \|_{HS}^q \lesssim \sum_{i=1}^{2}\| \widehat{f}(\pi_{(\xi , \lambda)}) (\pi_{(\xi , \lambda)}(h_i) - I_{d_\lambda}) \|_{HS}^q,$$for any $1 \leq q <\infty.$
    \item Now assume that the statement of the lemma is true for any compact nilpotent Vilenkin group of nilpotency class less than $\mathscr{n}$. Let $G$ have nilpotency class $\mathscr{n} = \mathscr{n} (G)$. Then again we have two cases for the unitary irreducible representations of $G$: either they are trivial in the center and they reduce to a representation of the nilpotent group $G/ \mathcal{Z}(G),$ or they are not trivial in the center and they are induced by a certain central character $\chi_\lambda$. Let us denote by $\xi$ the representations of the first kind and by $\pi_\lambda$ the representations of the second kind. Then once again, all the unitary irreducible representations of $G$ have the form $\xi \otimes \pi_\lambda.$ If $\chi_\lambda$ is not the trivial character so that $\pi_\lambda$ is not the identity representation, then for $h_\mathscr{n} \in \mathcal{Z}(G) \cap (G_n \setminus G_{n+1})$ we get :$$ \xi \otimes \pi_\lambda (h_{\mathscr{n}}) - I_{d_\xi d_\lambda}= (\chi_{\lambda} (h_{\mathscr{n}}) - 1) I_{d_\xi d_\lambda}, $$and once again this implies that $$|G/G_n| \langle \xi \otimes \pi_\lambda \rangle_{\mathscr{G}} \| \widehat{f} (\xi \otimes \pi_\lambda) \|_{HS} \lesssim \| \widehat{f} (\xi \otimes \pi_\lambda) (\xi \otimes \pi_\lambda (h_{\mathscr{n}}) - I_{d_\xi d_\lambda})\|_{HS}.$$
    Now we only need to deal with the representations of the form $\xi \otimes 1$. This are exactly the unitary irreducible representations of $G$ that descent to a representation of the quotient group $H:= G/ \mathcal{Z}(G)$, which can be made a compact Vilenkin group with the sequence of compact open subgroups $\mathscr{H} := \{ H_n \}_{n \in \N_0}$ given by $$H_n:= G_n / \mathcal{Z}(G_n).$$Clearly $H$ is of nilpotency class $\mathscr{n} -1$ so, by the induction hypothesis, there are $\overline{h}_1,.., \overline{h}_{\mathscr{n} - 1} \in H= G/ \mathcal{Z}(G)$ such that $|\overline{h}_1|_{\mathscr{H}} = ... = |\overline{h}_{\mathscr{n} -1}|_{\mathscr{H}} = |H/ H_n|^{-1}$ and   $$|H/H_n|^q \langle \xi \otimes 1 \rangle_{\mathscr{H}}^{-q} \| \widehat{f} (\xi \otimes 1 ) \|_{HS}^q \lesssim \sum_{i=1}^{\mathscr{n} - 1} \| \widehat{f} (\xi \otimes 1) (\xi \otimes 1 (h_{i}) - I_{d_\xi })\|_{HS}^q.$$To conclude the proof just notice that  $|\overline{h}_i|_{\mathscr{H}} = |H/H_n |^{-1}$ if and only if $|h_i|_{\mathscr{G}} = |G/G_n |^{-1}$, and also $\langle \xi \otimes 1 \rangle_{\mathscr{H}} = |H/H_m|$ if and only if $\langle \xi \otimes 1 \rangle_{\mathscr{G}} = |G/G_m|$. Adding the fact that   $$|H/H_n |= \frac{|G/G_n|}{|\mathcal{Z}(G/G_n)|},$$we finally obtain $$ \frac{|G/G_n|^q}{|\mathcal{Z}(G/G_n)|^q} |\mathcal{Z}(G/G_m)|^q \langle \xi \otimes 1 \rangle_{\mathscr{G}}^{-q} \| \widehat{f} (\xi \otimes 1 ) \|_{HS}^q \lesssim \sum_{i=1}^{\mathscr{n} - 1} \| \widehat{f} (\xi \otimes 1) (\xi \otimes 1 (h_{i}) - I_{d_\xi })\|_{HS}^q,$$and finally, since $$\frac{|\mathcal{Z}(G/G_m)|}{|\mathcal{Z}(G/G_n)|} \geq 1,$$we can conclude $$|G/G_n|^q \langle \xi \otimes \pi_\lambda \rangle_{\mathscr{G}}^{-q} \| \widehat{f} (\xi \otimes \pi_\lambda) \|_{HS}^q \lesssim \sum_{i = 1}^{\mathscr{n} (G)} \| \widehat{f} (\xi \otimes \pi_\lambda) (\xi \otimes \pi_\lambda (h_{i}) - I_{d_\xi d_\lambda})\|_{HS}^q.$$This concludes the proof. 
\end{itemize}
\end{proof}

Before diving into the proof of Theorems \ref{firsttitchmarshtheorem} and \ref{seconteo} we need to establish some preliminary results. We will start with the following proposition which will be useful in the proof of Theorem \ref{firsttitchmarshtheorem}.
\begin{pro}\label{propaux}
Let $\sigma :\widehat{G}\to \bigcup_{[\xi] \in \widehat{G}} \C^{d_\xi \times d_\xi}$ be such that $\sigma(\xi) \in \C^{d_\xi \times d_\xi}$ for every $[\xi] \in \widehat{G}$. Let $1 \leq r < \infty$  and $ \gamma > 0$ be positive real numbers. Then $$\langle \xi \rangle_\mathscr{G}^{\gamma} \sigma (\xi) \in L^{r} (\widehat{G}) \implies \sigma (\xi) \in L^\beta (\widehat{G}), $$ for all $ \frac{r}{1+\gamma r} < \beta < \infty$.
\end{pro}

\begin{proof}
Since $\langle \xi \rangle_\mathscr{G} \geq 1$, $$\langle \xi \rangle_\mathscr{G}^\gamma \sigma (\xi) \in L^{r} (\widehat{G}) \implies \sigma (\xi) \in L^\beta (\widehat{G}) ,$$for $\beta \geq r$. For $\beta < r$ it holds: \begin{align*}
    ||\sigma||_{L^\beta (\widehat{G})}^\beta &= \sum_{n \in \N_0} \sum_{[\xi] \in \widehat{G}_n} d_{\xi}^{\beta (\frac{2}{\beta} - \frac{1}{2})} ||\widehat{\sigma}||_{HS}^\beta = \sum_{n \in \N_0} \sum_{[\xi] \in \widehat{G}_n} d_\xi^{2} \Big( \frac{||\sigma (\xi)||_{HS}}{\sqrt{d_\xi}} \Big)^{\beta} \\ &= \sum_{n \in \N_0} \sum_{[\xi] \in \widehat{G}_n} d_\xi^2 \langle \xi \rangle_\mathscr{G}^{-\beta \gamma} \Big( \frac{\langle \xi \rangle_\mathscr{G}^\gamma || \sigma (\xi) ||_{HS}}{\sqrt{d_\xi}} \Big)^\beta \\ & = \sum_{n \in \N_0} \sum_{[\xi] \in \widehat{G}_n} d_\xi^{2 \frac{\beta}{r}} \Big( \frac{\langle \xi \rangle_\mathscr{G}^\gamma ||\widehat{\sigma} (\xi)||_{HS}}{\sqrt{d_\xi}}\Big)^\beta d_\xi^{2(1 - \frac{\beta}{r})} \langle \xi \rangle_\mathscr{G}^{- \beta \gamma} \\ & = \sum_{n \in \N_0} \sum_{[\xi] \in \widehat{G}_n} a_\xi b_\xi,
\end{align*}
where $$a_\xi := d_\xi^{2 \frac{\beta}{r}} \Big( \frac{\langle \xi \rangle_\mathscr{G}^\gamma ||\widehat{\sigma} (\xi)||_{HS}}{\sqrt{d_\xi}} \Big)^\beta, \esp \esp \text{and} \esp \esp b_\xi := d_\xi^{2(1 - \frac{\beta}{r})} \langle \xi \rangle_\mathscr{G}^{- \beta \gamma}.$$Now we use the H{\" o}lder inequality  to conclude \begin{align*}
     \sum_{n \in \N_0} \sum_{[\xi] \in \widehat{G}_n} d_\xi^{2} \Big( \frac{||\sigma (\xi)||_{HS}}{\sqrt{d_\xi}} \Big)^{\beta} &\leq \Big( \sum_{n \in \N_0} \sum_{[\xi] \in \widehat{G}_n} d_\xi^2 \Big( \frac{\langle \xi \rangle_\mathscr{G}^\gamma ||\widehat{\sigma} (\xi)||_{HS}}{\sqrt{d_\xi}} \Big)^{r} \Big)^{\beta/r} \times \Big( \sum_{n \in \N_0} \sum_{[\xi] \in \widehat{G}_n} d_\xi^2 \langle \xi \rangle_\mathscr{G}^{-\beta \gamma \frac{r}{r - \beta}}  \Big)^{1 - \beta/r}.
\end{align*} We can calculate explicitly $$\sum_{n \in \N_0} \sum_{[\xi] \in \widehat{G}_n} d_\xi^2 \langle \xi \rangle_\mathscr{G}^{-\beta \gamma \frac{r}{r - \beta}} = \sum_{n \in \N_0} |G/G_n|^{-\beta \gamma \frac{r}{r - \beta}} \sum_{[\xi] \in \widehat{G}_n}d_\xi^2 \leq \sum_{n \in \N_0} |G/G_n|^{1 -\beta \gamma \frac{r}{r - \beta}}, $$and the above series is convergent for $\beta > \frac{r}{1 + \gamma r}.$ This concludes the proof.
\end{proof}
Now we need an auxiliary lemma. It is known as the Duren lemma, and it is a Tauberian theorem that we will use to deduce some summability properties of functions defined in the unitary dual.

\begin{lema}\label{Durenlemma}
Suppose $c_i \geq 0$ and $0 < b <a$. then $$\sum_{i=1}^k i^a c_i = \mathcal{O} (N^b) \esp \text{as} \esp N \to \infty,$$if and only if $$\sum_{i=k+1}^\infty c_i = \mathcal{O} (N^{b-a}) \esp \text{as} \esp N \to \infty.$$
\end{lema}
Now we are in position to prove the first Titchmarsh theorem for compact Vilenkin groups satisfying Condition (A). 

\begin{proof}[Proof of Theorem \ref{firsttitchmarshtheorem}:]
From the local constancy of the matrix entries of the representations we know that  $\xi (h) - I_{d_\xi} = 0$ if $\langle \xi \rangle_\mathscr{G} \leq |h|_\mathscr{G}^{-1}$. From the Hausdorff-Young inequality we get $$\sum_{\langle \xi \rangle_\mathscr{G} > |h|_\mathscr{G}^{-1}} d^{q (\frac{2}{q} - \frac{1}{2})} \|  \widehat{f} (\xi)(\xi(h) - I_{d_\xi} )\|_{HS}^{q} =  ||\widehat{f}(\xi(h) - I_{d_\xi})||_{L^{q} (\widehat{G})}^{q} \leq || f(h \cdot) - f(\cdot)||_{L^p (\mu_G)}^{q} = \mathcal{O} (|h|^{\alpha q}_{\mathscr{G}}) .$$
Let $n \in \N_0$ be a natural number and let $\mathscr{n}$ and $h_1,..,h_{\mathscr{n}}$ be like in Condition (A). Then \begin{align*}
    \sum_{\langle \xi \rangle_\mathscr{G} > |G/G_n|}  d^{q (\frac{2}{q} - \frac{1}{2})} \langle \xi \rangle_{\mathscr{G}}^{-q} \| \widehat{f}(\xi) \|_{HS}^q \lesssim \sum_{\langle \xi \rangle_\mathscr{G} > |G/G_n|} \sum_{i=1}^{\mathscr{n}} d^{q (\frac{2}{q} - \frac{1}{2})} \|  \widehat{f} (\xi)(\xi(h_i) - I_{d_\xi} )\|_{HS}^{q} = \mathcal{O}(|G/G_n|^{- (\alpha+1) q}),
\end{align*}
for every $n \in \N_0$. Now with this information, taking $\gamma > \alpha$, we can apply Duren's lemma with $b = (\gamma - \alpha)q = (\gamma + 1)q - (\alpha + 1)q$ and $a= \gamma +1 $ to conclude that $$\sum_{\langle \xi \rangle_\mathscr{G} \leq |G/G_n|}  d^{q (\frac{2}{q} - \frac{1}{2})} \langle \xi \rangle_{\mathscr{G}}^{\gamma q} \| \widehat{f}(\xi) \|_{HS}^q = \mathcal{O} (|G/G_n|^{-(\gamma - \alpha)q}).$$Notice that we can apply Duren's lemma again in the reverse direction to obtain the first conclusion in Theorem \ref{firsttitchmarshtheorem}: $$\sum_{\langle \xi \rangle_\mathscr{G} > |G/G_n|}  d^{q (\frac{2}{q} - \frac{1}{2})}  \| \widehat{f}(\xi) \|_{HS}^q = \mathcal{O}(|G/G_n|^{-\alpha q}).$$

To continue with the proof let us define $$\Phi (k) := \sum_{\langle \xi \rangle \leq  |G/G_k|} d_\xi^{\beta (\frac{2}{\beta} - \frac{1}{2})} \langle \xi \rangle^{\gamma \beta } || \widehat{f} (\xi)||_{HS}^{\beta} .$$We can rewrite $$\Phi (k) = \sum_{ \langle \xi \rangle \leq |G/G_k|} d_\xi^2 \Big( \frac{\langle \xi \rangle^{\gamma} ||\widehat{f} (\xi)||_{HS}}{\sqrt{d_\xi}}\Big)^{\beta},$$and by H{\"o}lder inequality $$\Phi (k) \leq \Big( \sum_{ \langle \xi \rangle \leq |G/G_k|} d_\xi^2 \Big( \frac{\langle \xi \rangle^{\gamma} || \widehat{f} (\xi)||_{HS}}{ \sqrt{d_\xi}}  \Big)^{q} \Big)^{\beta / q} \Big( \sum_{ i=0}^k \sum_{[\xi] \in \widehat{G}_i} d_\xi^2 \Big)^{1 - \beta/ q}.$$Finally we use the fact that $$|G/G_n| =  \sum_{[\xi] \in \widehat{G/G_n} } d_\xi^2 = \sum_{k \leq n} \sum_{[\xi] \in \widehat{G}_k} d_\xi^2 ,$$ to obtain  $$\Phi (k) = \mathcal{O} (|G/G_k|^{(\gamma - \alpha)\beta}) \mathcal{O} (|G/G_k|^{1 - \beta / q}) = \mathcal{O} (|G/G_k|)^{(\gamma - \alpha)\beta + (1 - \beta/q)},$$and the above quantity is bounded when $\gamma < \alpha + 1/q$ and $$\beta \geq \frac{q}{(\alpha - \gamma)q + 1}.$$To conclude the proof just notice that $$\frac{\frac{q}{1 + (\alpha - \gamma)q}}{1 + \frac{\gamma q}{1 + (\alpha - \gamma)q}} = \frac{q}{ \alpha q +1},$$and use Proposition \ref{propaux}
\end{proof}

 With the arguments in the proof of Theorem \ref{firsttitchmarshtheorem}, the proof of Theorem \ref{seconteo}  is already halfway done.

\begin{proof}[Proof of Theorem \ref{seconteo}]
When $f \in Lip_{\mathscr{G}} (\alpha;2)$ then by the First Titchmarsh theorem we know that $$\sum_{\langle \xi \rangle_{\mathscr{G}} > |G/G_k|} d_\xi ||\widehat{f}(\xi)||_{HS}^2 = \mathcal{O} (|G/G_k|^{-2 \alpha}) \esp \esp \text{as k} \to \infty.$$Conversely, assume that $$\sum_{|G/G_k| < \langle \xi \rangle_{\mathscr{G}}} d_\xi ||\widehat{f} (\xi)||_{HS}^2 = \mathcal{O}(|G/G_k|^{-2 \alpha}).$$We know that $$||(\xi (h) - I_{d_\xi}) \widehat{f} (\xi) ||_{HS} \lesssim ||\widehat{f} (\xi) ||_{HS}.$$Therefore, writing $|h|_{\mathscr{G}} = |G/G_k|^{-1}$, we obtain: \begin{align*}
    ||f(h \cdot)- f (\cdot)||_{L^2 (\mu_G)}^2 &= \sum_{ |G/G_k| < \langle \xi \rangle_{\mathscr{G}} } d_\xi \| (\xi (h) - I_{d_\xi}) \widehat{f} (\xi) \|^2_{HS} \\ & \lesssim \sum_{|G/G_k| < \langle \xi \rangle_{\mathscr{G}}}   d_\xi ||\widehat{f} (\xi)||_{HS}^2 = \mathcal{O}(|h|_{\mathscr{G}}^{2 \alpha}).
\end{align*}This concludes the proof.
\end{proof}

\section{Proofs for Compact nilpotent $\ell$-adic Lie groups}
In this section, we adjust the necessary details for the proof Theorem \ref{firsttitchmarshtheoremLiegroups}. We follow again the same steps as in the proof of Theorem \ref{firsttitchmarshtheorem}, but this time Condition (A) need a subtle adjustment. For compact $\ell$-adic Lie groups the Lemma \ref{lemazerorep} takes the following form: 

\begin{lema}\label{lemazerorepLiegroups}
Let $G$ be a compact nilpotent $\ell$-adic Lie group with dimension $d$. Let $\mathscr{n} = \mathscr{n}(G)$ be the nilpotency class of $G$. Then for all $k\in \N_0$ there are $\mathscr{n}$ points $h_1,..., h_{\mathscr{n}} \in G$ satisfying $\|h_1 \|_{\ell} = ... = \|h_\mathscr{n}\|_\ell = \ell^{-k}$ and 
    $$ \ell^{kq} \langle \pi \rangle_{G}^{-q}    
\| \widehat{f}(\pi)  \|_{HS}^q \lesssim \sum_{i=1}^{\mathscr{n}}\| \widehat{f}(\pi) (\pi(h_i) - I_{d_\pi}) \|_{HS}^q,  \esp \esp 1 \leq q < \infty,  $$for every unitary irreducible representation $[\pi] \in \widehat{G}$, and any $ \widehat{f}(\pi) \in \C^{d_\pi \times d_\pi}$.
\end{lema}
\begin{pro}\label{propauxLiegroups}
Let $G$ be a $d$-dimensional compact nilpotent $\ell$-adic linear Lie group and let $\sigma :\widehat{G}\to \bigcup_{[\xi] \in \widehat{G}} \C^{d_\xi \times d_\xi}$ be such that $\sigma(\xi) \in \C^{d_\xi \times d_\xi}$ for every $[\xi] \in \widehat{G}$. Let $1 \leq r < \infty$  and $ \gamma > 0$ be positive real numbers. Then $$\langle \xi \rangle_G^{\gamma} \sigma (\xi) \in L^{r} (\widehat{G}) \implies \sigma (\xi) \in L^\beta (\widehat{G}), $$ for all $ \frac{r d}{d +\gamma r} < \beta < \infty$.
\end{pro}

\begin{proof}
Following the same arguments as in Proposition \ref{propaux} we get: 

\begin{align*}
    \| \sigma \|^\beta_{L^\beta (\widehat{G})} & \leq \Big( \sum_{n \in \N_0} \sum_{[\xi] \in \widehat{G}_n} d_\xi^2 \Big( \frac{\langle \xi \rangle_G^\gamma ||\widehat{\sigma} (\xi)||_{HS}}{\sqrt{d_\xi}} \Big)^{r} \Big)^{\beta/r} \times \Big( \sum_{n \in \N_0} \sum_{[\xi] \in \widehat{G}_n} d_\xi^2 \langle \xi \rangle_G^{-\beta \gamma \frac{r}{r - \beta}}  \Big)^{1 - \beta/r}, 
\end{align*}where we can calculate explicitly 
$$\sum_{n \in \N_0} \sum_{[\xi] \in \widehat{G}_n} d_\xi^2 \langle \xi \rangle_G^{-\beta \gamma \frac{r}{r - \beta}} = \sum_{n \in \N_0} \ell^{-n \beta \gamma \frac{r}{r - \beta}} \sum_{[\xi] \in \widehat{G}_n}d_\xi^2 = (1-\ell^{-d})\sum_{n \in \N_0} \ell^{n\big( d -\beta \gamma \frac{r}{r - \beta}\big)}. $$Clearly, the above series is convergent only when $ \frac{r d}{d +\gamma r} < \beta < \infty$.
\end{proof}

With the above, we can prove the first Titchmarsh theorem on compact $d$-dimensional nilpotent $\ell$-adic Lie groups. The proof of the second Titchmarsh theorem is exactly the same as for Theorem \ref{seconteo}, so we left to the reader the details. 

\begin{proof}[Proof of Theorem \ref{firsttitchmarshtheoremLiegroups}:]
Clearly $\xi (h) - I_{d_\xi} = 0$ if $\langle \xi \rangle_G \leq \|h\|_\ell^{-1}$. By the Hausdorff-Young inequality we have $$\sum_{\langle \xi \rangle_G > \|h\|_\ell^{-1}} d^{q (\frac{2}{q} - \frac{1}{2})} \| (\xi(h) - I_{d_\xi} ) \widehat{f} (\xi)\|_{HS}^{q} =  ||(\xi(h) - I_{d_\xi})\widehat{f}||_{L^{q} (\widehat{G})}^{q} \leq || f(h \cdot) - f(\cdot)||_{L^r (G)}^{q} = \mathcal{O} (\|h\|_\ell^{\alpha q}) .$$

Let $n \in \N_0$ be a natural number and let $\mathscr{n}$ and $h_1,..,h_{\mathscr{n}}$ be like in Lemma \ref{lemazerorepLiegroups}. Then \begin{align*}
    \sum_{\langle \xi \rangle_G > \ell^k}  d^{q (\frac{2}{q} - \frac{1}{2})} \langle \xi \rangle_{G}^{-q} \| \widehat{f}(\xi) \|_{HS}^q \lesssim \sum_{\langle \xi \rangle_G  > \ell^k } \sum_{i=1}^{\mathscr{n}} d^{q (\frac{2}{q} - \frac{1}{2})} \|  \widehat{f} (\xi)(\xi(h_i) - I_{d_\xi} )\|_{HS}^{q} = \mathcal{O}(\ell^{- k (\alpha+1) q}),
\end{align*}
for every $k \in \N_0$. Now with this information, taking $\gamma > \alpha$, we can apply once again Duren's lemma with $b = (\gamma - \alpha)q = (\gamma + 1)q - (\alpha + 1)q$ and $a= \gamma +1 $ to conclude that $$\sum_{\langle \xi \rangle_G \leq \ell^k}  d^{q (\frac{2}{q} - \frac{1}{2})} \langle \xi \rangle_G^{\gamma q} \| \widehat{f}(\xi) \|_{HS}^q = \mathcal{O} (\ell^{-(\gamma - \alpha)q}).$$W can apply Duren's lemma again in the reverse direction to obtain the first conclusion in Theorem \ref{firsttitchmarshtheoremLiegroups}: $$\sum_{\langle \xi \rangle_G > \ell^k}  d^{q (\frac{2}{q} - \frac{1}{2})}  \| \widehat{f}(\xi) \|_{HS}^q = \mathcal{O}(\ell^{-k\alpha q}).$$
Now, for $\beta \leq q$ using  H{\"o}lder inequality we get $$\sum_{\langle \xi \rangle_G \leq  \ell^k} d_\xi^{\beta (\frac{2}{\beta} - \frac{1}{2})} \langle \xi \rangle_G^{\gamma \beta } || \widehat{f} (\xi)||_{HS}^{\beta} \leq \Big( \sum_{ \langle \xi \rangle_G \leq \ell^k} d_\xi^2 \Big( \frac{\langle \xi \rangle^{\gamma} || \widehat{f} (\xi)||_{HS}}{ \sqrt{d_\xi}}  \Big)^{q} \Big)^{\beta / q} \Big(  \sum_{\langle \xi \rangle_G \leq \ell^k} d_\xi^2 \Big)^{1 - \beta/ q}.$$Finally we use the fact that $$\ell^{kd} =  \sum_{\langle \xi \rangle_G \leq \ell^k} d_\xi^2 ,$$to obtain  $$\sum_{\langle \xi \rangle_G \leq  \ell^k} d_\xi^{\beta (\frac{2}{\beta} - \frac{1}{2})} \langle \xi \rangle_G^{\gamma \beta } || \widehat{f} (\xi)||_{HS}^{\beta} = \mathcal{O} (\ell^{k(\gamma - \alpha)\beta}) \mathcal{O} (\ell^{kd(1 - \beta / q)}) = \mathcal{O} (\ell^{k(\gamma - \alpha)\beta + kd(1 - \beta/q)}).$$ The above quantity is bounded when $$\beta \geq \frac{q d}{(\alpha - \gamma)q + d}= \frac{d}{(\alpha - \gamma) +d - \frac{d}{p}}.$$To conclude the proof just notice that $$\frac{d \frac{qd}{d + (\alpha - \gamma)q}}{d + \frac{\gamma q d}{d + (\alpha - \gamma)q}} = \frac{q d}{ \alpha q +d} = \frac{dp}{\alpha p + dp - d},$$and use Proposition \ref{propauxLiegroups}
\end{proof}
\begin{rem}
    Notice how our results differ from those in \cite{DaherDelgadoRuzhansky}. Here a parameter $\gamma$ appears in the theorem, in contrast with \cite[Theorem 3.2]{DaherDelgadoRuzhansky}. The reason is that the authors in \cite{DaherDelgadoRuzhansky} overlooked the fact that, in order to obtain the condition $$\beta \geq \frac{q d}{(\alpha - 1)q + d},$$it must hold that $1< \alpha + d/q$. This means that \cite[Theorem 3.2]{DaherDelgadoRuzhansky} does not hold for arbitrary $\alpha , q$ and $d$, because one of this parameter is going to be dependent on the others. Also, it is desirable to have a result telling how much Sobolev regularity do Lipchitz functions of order $\alpha$ have, which is not cleat from the arguments in \cite{DaherDelgadoRuzhansky}. Considering that, we introduce the parameter $\gamma$ in Theorem \ref{firsttitchmarshtheoremLiegroups}. In this way $\alpha , q$ and $d$  can be choosen arbitrarily, and at the same time, the maximun value that this parameter can take tell us how much Sobolev regularity does $f \in Lip_G (\alpha ; 2)$ has. This is an improvement over \cite{DaherDelgadoRuzhansky}.
    
\end{rem}

\section{Modulus of continuity}

\begin{proof}[Proof of Theorem \ref{ProPlatonov1}:]
    In one hand, let us take $h \in G_n / G_{n+1}$. Then \begin{align*}
        \| f(h \cdot) - f(\cdot) \|_{L^2 (\mu_G)}^2 &= \sum_{\langle \xi \rangle_{\mathscr{G}} > |G/G_n|} d_\xi \| \widehat{f}(\xi)(\xi(h) - I_{d_\xi}) \|_{HS}^2 \\ & \leq 4 \sum_{\langle \xi \rangle_{\mathscr{G}} > |G/G_n|} d_\xi \| \widehat{f}(\xi) \|_{HS}^2,
    \end{align*}
so that $$\omega_2 (f , \mathscr{G}, n):= \sup_{h \in G_n} \| f(h \cdot) - f(\cdot) \|_{L^2 (\mu_G)} \leq 2  \Big( \sum_{\langle \xi \rangle_{\mathscr{G}} > |G/G_n|} d_\xi \| \widehat{f}(\xi) \|_{HS}^2 \Big)^{1/2}.$$
Conversely, if we integrate $\| f(h \cdot) - f(\cdot) \|_{L^2 (\mu_G)}^2$ with respect to $h$ we get $$\int_{G_n} \| f(h \cdot) - f(\cdot) \|_{L^2 (\mu_G)}^2 d \mu_G(h) = \sum_{\langle \xi \rangle_{\mathscr{G}} > |G/G_n|} d_\xi \int_{G_n} \| \widehat{f}(\xi) (\xi(h) - I_{d_\xi})   \|_{HS}^2 d \mu_G(h) .$$For every $[\xi] \in \widehat{G}$ and every $h \in G_n$ we can choose a basis $v_1,...,v_{d_\xi}$ of the representation space $\mathcal{H}_\xi$ where  $\xi(h) - I_{d_\xi}$ is a diagonal matrix. Since $\xi(h)$ is unitary, the elements of the diagonal have to be of the form $\zeta_j (h) - 1$, $1 \leq j \leq d_\xi$ where $|\zeta (h) |=1,$ and from Proposition \ref{proIntegralRep} we can deduce $$ \int_{G_n} \zeta_j (h)   d \mu_G(h) =0, \esp \esp 1 \leq j \leq d_\xi.$$In this way we obtain \begin{align*}
    \int_{G_n} \| \widehat{f}(\xi) (\xi(h) - I_{d_\xi})   \|_{HS}^2 d \mu_G(h) & = \sum_{j =1}^{d_\xi} \big( \int_{G_n} |\zeta_j (h) - 1|^2 d \mu_G(h) \big) \| \widehat{f}(\xi) v_j \|_{\mathcal{H}_\xi}^2 \\ & =\sum_{j =1}^{d_\xi} \big( \int_{G_n} 2(1 - \mathfrak{Re}[\zeta_j (h)]) d \mu_G(h) \big) \| \widehat{f}(\xi) v_j \|_{\mathcal{H}_\xi}^2 \\ &= 2|G_n|\sum_{j =1}^{d_\xi} \| \widehat{f}(\xi) v_j \|_{\mathcal{H}_\xi}^2 =  2|G_n| \| \widehat{f}(\xi) \|_{HS}^2.
\end{align*}
In conclusion, we proved: \begin{align*}
    2 |G_n| \sum_{\langle \xi \rangle_{\mathscr{G}} > |G/G_n|} d_\xi \| \widehat{f}(\xi) \|_{HS}^2 &= \int_{G_n} \| f(h \cdot) - f(\cdot) \|_{L^2 (\mu_G)}^2 d \mu_G(h) \\ & \leq \int_{G_n} \sup_{h \in G_n} \| f(h \cdot) - f(\cdot) \|_{L^2 (\mu_G)}^2 d \mu_G(h) = |G_n| \omega_2 (f , \mathscr{G}, n)^2,
\end{align*}
so we arrive at the conclusion $$ \Big(  \sum_{\langle \xi \rangle_{\mathscr{G}} > |G/G_n|} d_\xi \| \widehat{f}(\xi) \|_{HS}^2 \Big)^{1/2} \leq \frac{1}{\sqrt{2}} \omega_2 (f , \mathscr{G}, n).$$Finally, to see that the constant $1/\sqrt{2}$ is sharp, consider the characteristic functions $\1_{G_m}$. For this function it holds:
\[
\int_{G} | \1_{G_m} (hx) - \1_{G_m} (x) |^2 d \mu_G(x)= \begin{cases}
 0  \esp & \esp \text{if} \esp \esp h \in G_m, \\
2|G_m| \esp & \esp \text{if} \esp \esp h \notin G_m,  
\end{cases}
\]and also \[
\widehat{\1}_{G_m} (\xi) = \begin{cases}
 0_{d_\xi}  \esp & \esp \text{if} \esp \esp \langle \xi \rangle_{\mathscr{G}} > |G/G_m|, \\
|G_m| I_{d_\xi} \esp & \esp \text{if} \esp \esp \langle \xi \rangle_{\mathscr{G}} \leq |G/G_m|.  
\end{cases}
\]
So, in one direction we obtain: 
\[
\omega_2 (\1_{G_m}, \mathscr{G}, n ) := \sup_{h \in G_n} \| \1_{G_m} (h\cdot) - \1_{G_m} (\cdot) \|_{L^2 (\mu_G)} = \begin{cases}
 0  \esp & \esp \text{if} \esp \esp n \geq m, \\
\sqrt{2|G_m|} \esp & \esp \text{if} \esp \esp n<m,  
\end{cases}
\]
and in the other: \[
 \sum_{ \langle \xi \rangle_{\mathscr{G}} > |G/G_n|} d_\xi \| \widehat{\1}_{G_m} (\xi)\|^2_{HS} = \begin{cases}
 0  \esp & \esp \text{if} \esp \esp n \geq m, \\
|G_m| \big(1 - \frac{|G/G_n|}{|G/G_m|} \big) \esp & \esp \text{if} \esp \esp n<m.  
\end{cases}
\]
Summing up this means that: 
$$\Big( \sum_{ \langle \xi \rangle_{\mathscr{G}} > |G/G_n|} d_\xi \| \widehat{\1}_{G_m} (\xi)\|^2_{HS} \Big)^{1/2} = \frac{1}{\sqrt{2}} \omega_2 (\1_{G_m}, \mathscr{G}, n ) \big(1 - \frac{|G/G_n|}{|G/G_m|} \big).$$By fixing $n$ and letting $m \to \infty$ we arrive to the conclusion that the constant $1/\sqrt{2}$ is sharp. 
\end{proof}

\section{Dini–Lipschitz functions}

In this section, we sketch the proofs of Theorems \ref{toeDL1} and \ref{teoDL2} on compact Vilenkin groups satisfying Condition (A). For the proof of Theorem \ref{toeDL1} and Theorem \ref{teoDL2} we need the following modified version of Duren's lemma.  See \cite[Lemma 4.1]{DaherDelgadoRuzhansky} for the proof. 

\begin{lema}\label{lemmaDurenModified}
 Suppose $\nu \in \R$, $c_k \geq 0$ and $0 < b <a.$ Then 
$$\sum_{k = 1}^{N} k^a c_k = \mathcal{O} ( N^b (\log{N})^{\nu} ),$$
if and only if $$\sum_{k = N}^{\infty} c_k = \mathcal{O}( N^{b-a}( \log{N})^\nu  ).$$ 
\end{lema}

The proof of Theorem \ref{toeDL1} is almost the same as in Theorem \ref{firsttitchmarshtheorem} so here we only include the necessary modifications. 

\begin{proof}[Proof of Theorem \ref{toeDL1}:]
Let $f \in DL_\mathscr{G} (\alpha, \nu ; p)$. Then $$\sum_{\langle \xi \rangle_\mathscr{G} > |h|_\mathscr{G}^{-1}} d^{q (\frac{2}{q} - \frac{1}{2})} \|  \widehat{f} (\xi)(\xi(h) - I_{d_\xi} )\|_{HS}^{q} = \mathcal{O} (|h|_\mathscr{G}^{\alpha q} \big( \log{\frac{1}{|h|_{\mathscr{G}}}} \big)^{\nu q}) .$$
Let $n \in \N_0$ be a natural number and let $\mathscr{n}$ and $h_1,..,h_{\mathscr{n}}$ be like in Condition (A). Then \begin{align*}
    \sum_{\langle \xi \rangle_\mathscr{G} > |G/G_n|}  d^{q (\frac{2}{q} - \frac{1}{2})} \langle \xi \rangle_{\mathscr{G}}^{-q} \| \widehat{f}(\xi) \|_{HS}^q &\lesssim \sum_{\langle \xi \rangle_\mathscr{G} > |G/G_n|} \sum_{i=1}^{\mathscr{n}} d^{q (\frac{2}{q} - \frac{1}{2})} \|  \widehat{f} (\xi)(\xi(h_i) - I_{d_\xi} )\|_{HS}^{q} \\ & = \mathcal{O}(|G/G_n|^{- (\alpha+1) q} (\log{|G/G_n|})^{\nu q}),
\end{align*}
for every $n \in \N_0$. Now we apply Lemma \ref{lemmaDurenModified} with $b = (\gamma - \alpha)q = (\gamma + 1)q - (\alpha + 1)q$ and $a= \gamma +1 $ to conclude that $$\sum_{\langle \xi \rangle_\mathscr{G} \leq |G/G_n|}  d^{q (\frac{2}{q} - \frac{1}{2})} \langle \xi \rangle_{\mathscr{G}}^{\gamma q} \| \widehat{f}(\xi) \|_{HS}^q = \mathcal{O} (|G/G_n|^{-(\gamma - \alpha)q} (\log{|G/G_n|})^{\nu q}).$$And again, we use the lemma in the reverse direction to obtain: $$\sum_{\langle \xi \rangle_\mathscr{G} > |G/G_n|}  d^{q (\frac{2}{q} - \frac{1}{2})}  \| \widehat{f}(\xi) \|_{HS}^q = \mathcal{O}(|G/G_n|^{-\alpha q}(\log{|G/G_n|})^{ \nu q}).$$Following the same steps as in the proof of Theorem \ref{firsttitchmarshtheorem}, we arrive to $$\sum_{\langle \xi \rangle \leq |G / G_n|} d_\xi^{\beta (\frac{2}{\beta} - \frac{1}{2})} \langle \xi \rangle_{\mathscr{G}}^{\beta \gamma} \| \widehat{f}(\xi) \|_{HS}^{\beta} = \mathcal{O} ( |G/G_n|^{(\gamma - \alpha)\beta + (1 - \frac{\beta}{q})} (\log{|G/G_n|})^{\nu \beta} ), \esp \esp \text{as} \esp n \to \infty. $$ To conclude the proof just notice that the above quantity is bounded provided either $$ \esp  \frac{q}{(\alpha - \gamma)q + 1} = \frac{p}{(\alpha - \gamma)p + p - 1} \leq \beta \leq q , \esp \esp \text{and} \esp \esp \nu = 0,$$ or$$ \esp  \frac{q}{(\alpha - \gamma)q + 1} = \frac{p}{(\alpha - \gamma)p + p - 1} < \beta \leq q , \esp \esp \text{and} \esp \esp \nu \in \R.$$This fact together with Proposition \ref{propaux} concludes the proof. 
\end{proof}

\begin{proof}[Proof of Theorem \ref{teoDL2}:]
In one hand, from the proof of Theorem \ref{toeDL1} we know that the condition $$\| f(\cdot h) - f(\cdot) \|_{L^2(\mu_G)} = \mathcal{O} \Big( |h|_\mathscr{G}^\alpha \big(\log{\frac{1}{|h|_\mathscr{G}}}\big)^{\nu} \Big),$$implies that:
$$\sum_{\langle \xi \rangle_\mathscr{G} > |G/G_n|}  d^{q (\frac{2}{q} - \frac{1}{2})}  \| \widehat{f}(\xi) \|_{HS}^q = \mathcal{O}(|G/G_n|^{- 2 \alpha }(\log{|G/G_n|})^{ 2\nu }).$$On the other hand, by Theorem \ref{ProPlatonov1}, the above condition implies that $$\| f(\cdot h) - f(\cdot) \|_{L^2(\mu_G)} = \mathcal{O} \Big( |h|_\mathscr{G}^\alpha \big(\log{\frac{1}{|h|_\mathscr{G}}}\big)^{\nu} \Big).$$This concludes the proof.

\end{proof}

\nocite{*}
\bibliographystyle{acm}
\bibliography{main}
\Addresses

\end{document}